%% file: Quasi_F_split_non_uniruled.tex
\input{macros}

\author{Jefferson Baudin}
\date{}

\usepackage[margin=.91in]{geometry}
\setlist{  
	listparindent=\parindent,
	parsep=0pt,
}

\subjclass[2020]{14G17, 14E30, 14J32}
\keywords{quasi-$F$-splittings, uniruledness, log canonical singularities, $K$-trivial varieties}

\title[Global and local geometry of $K$--trivial quasi--$F$--split varieties]{On the global and local geometry of quasi--$F$--split varieties with trivial canonical bundle}

\begin{document}
	\maketitle
	\begin{abstract}
		We solve certain questions related to the geometry and singularities of quasi-$F$-split varieties with trivial canonical bundle. Namely, we prove that:
		\begin{enumerate}
			\item regular quasi-$F^{\infty}$-split varieties are not geometrically uniruled (this generalizes and significantly simplifies the earlier results of \cite{Patakfalvi_Zdanowicz_Ordinary_varieties_with_trivial_canonical_bundle_are_not_uniruled}) and have geometrically canonical singularities;
			\item there exist quasi-$F$-split surfaces with trivial canonical bundle which are not quasi-$F^{\infty}$-split, answering negatively a question raised in \cite{Quasi_F_splittings_III}. 
			\item normal quasi-$F$-split varieties with trivial canonical bundle are geometrically normal (this extends a result of \cite{Kawakami_Takamatsu_Yoshikawa_Fedder_type_criteria_for_quasi_F_splittings_II});
			\item quasi-$F^e$-pure normal varieties $X$ such that  $mp^eK_X$ is Cartier for $m$ coprime to $p$ are log canonical, under a resolution of singularities hypothesis.
		\end{enumerate} 
	\end{abstract}

	%\tableofcontents
	
\section{Introduction}

\subsection{Main results}

In positive characteristic, it is common that algebraic varieties depict unusual behaviour from a characteristic zero point of view, already in the case of $K$-trivial varieties. For example, certain K3 surfaces can be unirational, and quasi-bielliptic surfaces admit fibrations where all fibers are singular. One way to prevent these pathological behaviours is to require that our varieties are $F$-split \cite{Mehta_Ramanathan_Frobenius_splitting_and_cohomology_vanishing_for_Schubert_varieties}. Indeed $F$-split $K$-trivial varieties are never uniruled \cite{Patakfalvi_Zdanowicz_Ordinary_varieties_with_trivial_canonical_bundle_are_not_uniruled}, and general fibers of their fibrations can have at worst mild singularities \cite[Theorem 6.3]{Patakfalvi_Zdanowicz_Beauville_Bogomolov_in_pos_char}. Nevertheless, one may argue that $F$-splittings are too restrictive to prevent such pathologies, and that the broader class of quasi-$F$-split varieties (\cite{Yobuko_Quasi_F_splitting_and_lifting_of_CY_varieties, Quasi_F_splittings_I, Quasi_F_splittings_II, Quasi_F_splittings_III, Quasi_F^e_splittings_and_quasi-F_regularity}) should also behave well from a characteristic zero point of view. A good toy example to see this is the case of a K3 surface $S$, depending on its Artin-Mazur height $h(S) \in \bN \cup \{\infty\}$: it can be $F$-split (i.e. $h(S) = 1$), quasi-$F$-split (i.e. $h(S) < \infty$) or supersingular (i.e. $h(S) = \infty$). Since $H^2(S, W\cO_{S, \bQ}) \neq 0$ for a quasi-$F$-split K3 surface \cite{Yobuko_Quasi_F_splitting_and_lifting_of_CY_varieties}, it follows from the methods of \cite{Patakfalvi_Zdanowicz_Ordinary_varieties_with_trivial_canonical_bundle_are_not_uniruled} that they cannot be uniruled, while supersingular K3 surfaces might actually be! It is therefore natural to try to extend these types of result to the quasi-$F$-split realm, and this is one of the objectives of this paper.

From now on, a variety means an integral, separated scheme of finite type over a field $K$ that is finitely generated over a perfect field $k$ of characteristic $p > 0$ (hence $K$ might \textbf{not} be perfect). The following result is a generalization of \cite{Patakfalvi_Zdanowicz_Ordinary_varieties_with_trivial_canonical_bundle_are_not_uniruled}, and our proof is also significantly simpler.

\begin{thm_letter}\label{main_thm_qFsplit_non_uniruled}
	Let $X$ be a regular proper variety such that $K_X \equiv 0$ and $H^0(X, \cO_X) = K$. If $X$ is quasi-$F^{\infty}$-split, then it is not geometrically uniruled and has geometrically canonical singularities. 
\end{thm_letter}

We also have a version of \autoref{main_thm_qFsplit_non_uniruled} admitting mild singularities, see \autoref{cor_main_thm_regular_case}. We do not know whether requiring mere quasi-$F$-splitting instead of quasi-$F^{\infty}$-splitting is enough to obtain non-uniruledness. A first question one may ask in this direction is the following: does there exist a quasi-$F$-split $K$-trivial variety which is not quasi-$F^{\infty}$-split? The local counterpart of this question was raised in \cite{Quasi_F_splittings_III} (after their Theorem H), to which we give the following negative answer:

\begin{thm_letter}\label{intro_example}
	Let $S$ be a smooth projective quasi-$F$-split surface with $\omega_S\cong \cO_S$ over an algebraically closed field of characteristic $p > 0$. Then the following holds:
	
	\begin{enumerate}
		\item if $p \neq 2$, then $S$ is quasi-$F^{\infty}$-split;
		\item if $p = 2$, then $S$ might not be quasi-$F^2$-split.
	\end{enumerate}
	
	Consequently, there exists a quasi-$F$-pure, log canonical threefold singularity $X$ such that $K_X$ is Cartier, but which is not quasi-$F^{\infty}$-pure.
\end{thm_letter}

We conjecture that for any $d \geq 1$, there exists an integer $p_d > 0$ such that if $p > p_d$, any smooth proper quasi-$F$-split $K$-trivial variety of dimension $d$ over an algebraically closed field of characteristic $p$ is quasi-$F^{\infty}$-split. Even though this problem is certainly very hard in higher dimensions, we hope that it is solvable in dimension $3$. \\
%	As a matter of fact, there exists a unique class of $K$-trivial smooth surfaces $S$ over an algebraically closed field which is quasi-$F$-split but not quasi-$F^{\infty}$-split (in fact, it is not $n$-quasi-$F^2$-split for any $n \geq 1$). These surfaces are explicitly given by the $\bZ/2\bZ$-quotient of $E_1 \times E_2$, where $E_1$ is an ordinary elliptic curve acted on by a $2$-torsion point, and $E_2$ is a supersingular elliptic curve acted on by the inverse involution (note that this quotient is not uniruled, so there are no counterexamples to our question in dimension $2$). The relevant computations and proofs will appear in a note of Kawakami, Takagi and Yoshikawa.
%\end{rem_blank}

As a result of our methods, we also prove the following result, generalizing \cite[Theorem 4.13]{Kawakami_Takamatsu_Yoshikawa_Fedder_type_criteria_for_quasi_F_splittings_II}:
	
\begin{thm_letter}\label{intro_qF_split_implies_geom_normal}
	Let $X$ be a normal proper variety such that $K_X \equiv 0$ and $H^0(X, \cO_X) = K$. If $X$ is quasi-$F$-split, then $X$ is geometrically normal.
\end{thm_letter}

\begin{rem_blank}
	Recall that normal proper varieties $X$ which are globally $F$-split have geometrically canonical (and even strongly $F$-regular) singularities, even without the $K$-trivial assumption thanks to \cite[Theorem 6.3]{Patakfalvi_Zdanowicz_Beauville_Bogomolov_in_pos_char}, so our $K$-triviality assumption might seem odd at a first glance. However, this assumption is necessary, since for example all regular conics are quasi-$F$-split by \cite[Corollary 5.16]{Quasi_F_splittings_I}, while they may not even be geometrically reduced when $p = 2$.
\end{rem_blank}

Let us now move to a local counterpart of quasi-$F$-split $K$-trivial varieties, namely quasi-$F$-pure singularities $X$ such that $K_X$ is $\bQ$-Cartier. It was recently proven in \cite[Theorem A]{Sato_Takagi_Yoshikawa_Quasi_F_splittings_versus_log_canonicity} (see also \cite{Kawakami_Takagi_Yoshikawa_Quasi_F_singularities_and_singularities_in_birational_geometry}) that quasi-$F^{\infty}$-pure normal singularities are log canonical. In this paper, we tackle this question in the case of quasi-$F^e$-purity:

\begin{thm_letter}\label{intro_qFsplit_implies_lc}
	Let $R$ be a normal, quasi-$F^e$-pure local ring such that $mp^eK_R$ is Cartier for some $m > 0$ coprime to $p$. If a prime-to-$p$ index cover of $\Spec(R)$ admits a log resolution, then $R$ has log canonical singularities.
\end{thm_letter}

We refer the reader to \autoref{qF_pure_implies_lc_improved} and the paragraph right before to explain what we mean by a ``prime-to-$p$ index cover''. As suggested by Shunsuke Takagi, we can extend \autoref{intro_qFsplit_implies_lc} to the case of quasi-$F$-split pairs $(R, \Delta)$ where $\Delta$ has standard coefficients and $K_R + \Delta$ is $\bZ_{(p)}$-Cartier by \cite[Proposition 3.21]{Sato_Takagi_Yoshikawa_Quasi_F_splittings_versus_log_canonicity}, under the assumption that the index-one cover in  \emph{loc. cit.} admits a log resolution.

\subsection{Acknowledgments}

I would like to thank Tatsuro Kawakami, Alapan Mukhopadhyay, Shunsuke Takagi, Hiromu Tanaka, Jakub Witaszek, Zheng Xu, Fuetaro Yobuko, and Yang Zhang for interesting conversations related to the content of this article.  Financial support was provided by the grant \#2000201-231484 from the Swiss National Science Foundation.

\section{Notations and definitions}
	Throughout the paper, we fix a perfect field $k$ of characteristic $p > 0$, and a finitely generated field $K$ over $k$. A \emph{variety} denotes an integral, separated scheme of finite type over $K$. A \emph{big open} subset is an open subset whose complement has codimension at least $2$. For notions of singularities of the minimal model program (e.g. canonical, klt and log canonical singularities), we refer the reader to \cite{Kollar_Singularities_of_the_MMP}.
		
	Fix integers $n, e \geq 1$, and let $X$ be a variety over $K$, with induced morphism $\pi \colon X \to \Spec(k)$. We will denote by $K(X)$ the field of fractions of $X$, and by $W_nX$ the scheme given by the locally ringed space $(X, W_n\cO_X)$, where $W_n\cO_X$ denotes the $n$--truncated Witt vectors over $X$. The induced morphism $\pi \colon W_nX \to \Spec (W_n(k))$ is therefore separated and essentially of finite type (\cite[Propositions A.1 and A.6]{Langer_Zink_De_Rham_Witt_cohomology_for_a_proper_and_smooth_morphism}), so we can define the dualizing complex $W_n\omega_X^{\bullet} \coloneqq \pi^!\cO_{W_n(k)}$ on $X$ (see \cite[5.2]{Nayak_Compactification_for_essentially_finite_type_maps}). We set $W_n\omega_X$ to be the first non--zero cohomology sheaf of $W_n\omega_X^{\bullet}$. The natural closed immersion $X \inj W_nX$ induces by duality an injection $\omega_X \inj W_n\omega_X$, and throughout we will think of sections of $\omega_X$ as living in $W_n\omega_X$. Finally, we will denote the trace of the $e$'th iterate Frobenius $F^e \colon W_nX \to W_nX$ by $C^e_n \colon F^e_*W_n\omega_X \to W_n\omega_X$ (the morphism $F^e$ is finite, since it is proper by \cite[Corollary A.7]{Langer_Zink_De_Rham_Witt_cohomology_for_a_proper_and_smooth_morphism} and affine). Given an ideal $\cI \inc \cO_X$, we set $W_n\cI \coloneqq \set{(s_1, \dots, s_n) \in W_n\cO_X}{s_i \in \cI}$. Given a $\bZ$-divisor $D$, we set $W_n\omega_{(X, D)} \coloneqq \HHom(W_n\cI, W_n\omega_X)$ (see \cite[Notation 2.2.13 and below]{Baudin_Witt_GR_vanishing_and_applications}), and $W_n\cO_X(D) \coloneqq \set{(s_1, \dots, s_n) \in W_n(K(X))}{s_i \in \cO_X(p^{i - 1}D)}$. Note that if $D$ is a $\bZ$-divisor (resp. a Cartier divisor), then this latter sheaf is coherent and $S_2$ (resp. a line bundle) by \cite[Remark 3.6, Proposition 3.8 and Proposition 3.12]{Tanaka_Vanishing_theorems_of_Kodaira_type_for_Witt_canonical_sheaves}. 
	
	Given a torsion-free $S_2$-sheaf $\cM$ on $W_nX$ (e.g. $W_n\omega_X$ by \stacksproj{0AWE}) and a $\bZ$-divisor $D$, we define $\cM(D) \coloneqq j_*(\cM|_{X_{\reg}} \otimes W_n\cO_{X_{\reg}}(D|_{X_{\reg}}))$, where $j \colon X_{\reg} \inj X$ is the inclusion of the regular locus $X_{\reg}$ of $X$.
	
	\begin{sdefn}[{\cite{Quasi_F^e_splittings_and_quasi-F_regularity}}]\label{def_qF_split}
		Let $X$ be a normal variety. We set $(-)^* \coloneqq \HHom(-, W_n\omega_X(-K_X))$ (note that for a divisor $D$, then $\cO_X(D)^* = \cO_X(-D)$).
		\begin{enumerate}
			\item Given $n, e \geq 1$ and a divisor $D$ on $X$, we set $Q^e_{X, D, n}$ to be the pushout of the diagram 
			\[ \begin{tikzcd}
				W_n\cO_X(D) \arrow[d] \arrow[rr] &  & F^e_*W_n\cO_X(p^eD) \\
				\cO_X(D).
			\end{tikzcd} \] We also define $q^eS^0_n(X, \cO_X(D)) \coloneqq \im(H^0(X, (Q^e_{X, -D, n})^*) \to H^0(X, \cO_X(D)))$.
			\item We set $q^{\infty}S^0(X, D) \coloneqq \bigcap_{e > 0} \bigcup_{n > 0} q^eS^0_n(X, D) \inc H^0(X, \cO_X(D))$.
			\item We say that $X$ is $n$-quasi-$F^e$-split (resp. quasi-$F^{\infty}$-split) if $q^eS^0_n(X, \cO_X) = H^0(X, \cO_X)$ (resp. $q^{\infty}S^0(X, \cO_X) = H^0(X, \cO_X)$). We refer the reader to \cite[Proposition 3.19]{Quasi_F^e_splittings_and_quasi-F_regularity} for an alternate definition using ``splittings''.
			\item We say that $X$ is $n$-quasi-$F^e$-pure (resp. quasi-$F^{\infty}$-pure) if it is $n$-quasi-$F^e$-split (resp. quasi-$F^{\infty}$-split) affine locally.
		\end{enumerate}
	\end{sdefn}
	
\section{Non-uniruledness and geometric properties}

We will use the following standard result. 

\begin{slem}\label{BPRZ}
	Let $\pi \colon Y \to X$ be a projective birational morphism of normal varieties, let $D$ be a Cartier divisor on $X$, and let $E_1$ and $E_2$ be effective $\pi$-exceptional divisors on $Y$, with $E_2 \neq 0$. Then $\pi_*\cO_Y(\pi^*D + E_1 - E_2)$ is naturally a strict subsheaf of $\cO_X(D)$.
\end{slem}
\begin{proof}
	Let us include a proof for completeness (we will mimic the argument in the proof of \cite[Lemma 3.5]{Baudin_Patakfalvi_Rosler_Zdanowicz_On_Gorenstein_Q_p_rational_threefolds_and_fourfolds}). By the projection formula, we may assume that $D = 0$. First, note that $\pi_*\cO_Y(E_1 - E_2)$ is torsion-free, and is equal to $\pi_*\cO_Y = \cO_X$ on a big open subset. Since the latter sheaf is $S_2$, we obtain an inclusion $\pi_*\cO_Y(E_1 - E_2) \inc \cO_X$. However, this inclusion must be strict, since for example $1 \in H^0(X, \cO_X)$ does not belong to $\pi_*\cO_Y(E_1 - E_2)$.
\end{proof}

We will also need the following rephrasing of \autoref{def_qF_split} in the case of interest for us:

\begin{slem}\label{def_quasi_stable_sections}
	Let $X$ be a normal variety, and let $n, e \geq 1$. Then we have \[q^eS^0_n(X, \omega_X) = \set{ \eta \in H^0(X, \omega_X)}{ \eta = C^e_n(\psi) \emph{ for some } \psi \in H^0(X, F^e_*W_n\omega_X)}. \] 
\end{slem}
\begin{proof}
	Applying $(-)^* = \HHom(-, W_n\omega_X(-K_X))$ to the pushout square from the definition gives a pullback square
	\[ \begin{tikzcd}
		F^e_*W_n\omega_X \arrow[rr, "C^e_n"]              &  & W_n\omega_X              \\
		{(Q^e_{X, -K_X, n})^*} \arrow[u] \arrow[rr] &  & \omega_X \arrow[u, hook]
	\end{tikzcd} \]
	Since $q^eS^0_n(X, \omega_X)$ is by definition the image of the lower map on global sections, the result follows from the explicit definition of a pullback.
\end{proof}

%\begin{srem}\label{definitions_agree}
%	This definition agrees with \cite[Definitions 3.38 and 3.45]{Quasi_F^e_splittings_and_quasi-F_regularity} (the pushout in their definition is transformed into a pullback by taking their functor $(-)^*$, giving precisely our definition).
%\end{srem}

\begin{scor}\label{quasi-stable_maps_into_quasi-stable}
	Let $X$ be a normal variety, and let $n, e \geq 1$. Then $q^eS^0_n(X, \omega_X) \neq 0$ if and only if $C^e_n \colon H^0(X, F^e_*W_n\omega_X) \to H^0(X, W_n\omega_X)$ is non-zero. 
\end{scor}
\begin{proof}
	Throughout, we will use \autoref{def_quasi_stable_sections} without further mention. The ``left-to-right'' direction is then clear, so let us assume that $C^e_n \colon H^0(X, W_n\omega_X) \to H^0(X, W_n\omega_X)$ is non-zero. Since $q^eS^0_j(X, \omega_X) \inc q^eS^0_{j + 1}(X, \omega_X)$ for any $j \geq 1$, we may assume that $n$ is minimal for the above property. If $n = 1$, there is nothing to do, so assume that $n > 1$. The proof follows from the commutativity of the following diagram, and the fact that the bottom row is exact:	
	\[ \begin{tikzcd}
		&                         & {H^0(X, F^e_*W_n\omega_X)} \arrow[d, "C^e_n"] \arrow[r]  & {H^0(X, F^{e + 1}_*W_{n - 1}\omega_X)} \arrow[d, "C^e_{n - 1} = 0"] \\
		0 \arrow[r]  & {H^0(X, \omega_X)} \arrow[r]  & {H^0(X, W_n\omega_X)} \arrow[r]                          & {H^0(X, F_*W_{n - 1}\omega_X)}      \qedhere                               
	\end{tikzcd} \]
\end{proof}

Let us now move to the proof of \autoref{main_thm_qFsplit_non_uniruled}. Our approach will be to show that a normal proper variety $Y$ endowed with a generically finite and dominant morphism to some $X$ such that $q^{\infty}S^0(X, \omega_X) \neq 0$ satisfies $H^0(Y, \omega_Y) \neq 0$, and hence cannot admit a birational morphism to a normal proper variety of the form $Z \times \bP^1$. We start with the purely inseparable case.

\begin{slem}\label{purely_inseparable_case}
	Let $f \colon Y \to X$ be a finite morphism of normal varieties such that $K(Y)^{p^f} \inc K(X)$ for some $f > 0$. If $q^eS^0_n(X, \omega_X) \neq 0$ for some $n \geq 1$ and $e \geq f$, then $q^{e - f}S^0_n(Y, \omega_Y) \neq 0$ (we set $q^0S^0_n(X, \omega_X) \coloneqq H^0(X, \omega_X)$). In particular, $q^{\infty}S^0(X, \omega_X) \neq 0$ if and only if $q^{\infty}S^0(Y, \omega_Y) \neq 0$.
\end{slem}
\begin{proof}
	By normality, there exists a factorization \[ \begin{tikzcd}
		X \arrow[r, "g"] \arrow[rr, "F^f"', bend right] & Y \arrow[r, "f"] & X.
	\end{tikzcd} \] The same then holds for the schemes $W_nX$ and $W_nY$, so there is a factorization
	\[ \begin{tikzcd}
		F^f_*W_n\omega_X \arrow[rr] &  & f_*W_n\omega_Y \arrow[rr] &  & W_n\omega_X
	\end{tikzcd} \] of $C^f_n \colon F^e_*W_n\omega_X \to W_n\omega_X$. In particular, we obtain a commutative diagram \[ \begin{tikzcd}
		F^e_*W_n\omega_X \arrow[r] \arrow[rr, "C^f_n", bend 	left = 20] & F^{e - f}_*f_*W_n\omega_Y \arrow[r] \arrow[d, "f_*C^{e - f}_n"'] & F^{e - f}_*W_n\omega_X \arrow[d, "C^{e - f}_n"] \\
		& f_*W_n\omega_Y \arrow[r]                                      & W_n\omega_X.                                    
	\end{tikzcd} \] If $\eta \in H^0(X, F^e_*W_n\omega_X)$ satisfies $C^e_n(\eta) \neq 0$, then its image $\psi$ in $H^0(X, F^{e - f}_*f_*W_n\omega_Y) = H^0(Y, F^{e - f}_*W_n\omega_Y)$ satisfies $C^{e - f}_n(\psi) \neq 0$, so the result follows from \autoref{quasi-stable_maps_into_quasi-stable}.
\end{proof}

This is already enough to prove the following generalization of \cite[Theorem 4.13]{Kawakami_Takamatsu_Yoshikawa_Fedder_type_criteria_for_quasi_F_splittings_II}:

\begin{scor}[{\autoref{intro_qF_split_implies_geom_normal}}]\label{qFsplit_implies_geom_normal}
	Let $X$ be a normal, proper, quasi-$F$-split variety with $K_X \equiv 0$. If $H^0(X, \cO_X) = K$, then $X$ is geometrically normal. 
\end{scor}
\begin{proof}
	By the proof of \cite[Proposition 3.14]{Quasi_F_splittings_I} and our assumption that $K_X \equiv 0$, there exists $n > 0$ such that $(p^n - 1)K_X \sim 0$. Let $m > 0$ be the smallest integer such that $mK_X \sim 0$ (note that $m$ is coprime to $p$), and let $f \colon Y \to X$ denote the $\mu_m$-cover associated to the torsion $\bZ$-divisor $K_X$. Since $f$ is quasi-étale and $X$ is quasi-$F$-split, we deduce by \cite[Proposition 3.24]{Quasi_F^e_splittings_and_quasi-F_regularity} that $Y$ is quasi-$F$-split on a big open subset, and therefore quasi-$F$-split by the $S_2$ property. Hence, $q^1S^0_n(Y, \cO_Y) \neq 0$ for some $n > 0$ which we fix. Since $\cO_Y \cong \omega_Y$, we deduce that $q^1S^0_n(Y, \omega_Y) \neq 0$ too. Note that since $X$ is $S_2$, it is enough to verify that $Y$ is geometrically $R_1$ (which can be checked at codimension one points), so it is enough to show that $Y$ is geometrically normal (note that the field extension $H^0(Y, \cO_Y) \cni H^0(X, \cO_X) = K$ is separable, so we may assume that $H^0(Y, \cO_Y) = K$). If we assume by contradiction that $Y$ is not geometrically normal, $Y_{K^{1/p}}$ is not normal by \cite[Proposition 2.10]{Tanaka_Invariants_of_algebraic_varieties_over_imperfect_fields}, so by \cite[Theorem 1.1]{Ji_Waldron_Structure_of_geometrically_non_reduced_varieties} we obtain that $H^0(Z, \omega_Z) = 0$ where $Z$ is the normalisation of $(Y_{K^{1/p}})_{\red}$. Since $Z \to Y$ is purely inseparable of height one, this contradicts \autoref{purely_inseparable_case}.
\end{proof}

\begin{slem}\label{pullback_by_separable_maps}
	Let $f \colon Y \to X$ be a generically finite, dominant and separable morphism of normal varieties, and assume that $X$ is regular. Then there exist canonical injections $q^eS^0_n(X, \omega_X) \inj q^eS^0_n(Y, \omega_Y)$ for all $e, n \geq 1$, and a canonical injection $q^{\infty}S^0(X, \omega_X) \inj q^{\infty}S^0(Y, \omega_Y)$.
\end{slem}
\begin{proof}
	Since this can be verified in a big open subset of $Y$, we may assume that $Y$ is regular. Let $r$ be the size of a $p$-basis locally on $X$, and let us explain how to identify $W_n\omega_X$ with $W_n\Omega_X^r$ (namely the $r$'th piece of the absolute de Rham-Witt complex of \cite{Illusie_Complexe_de_de_Rham_Witt_et_cohomologie_cristalline}) and similarly on $Y$. Let $\cX'$ be a smooth, separated scheme of finite type over the perfect field $k$, whose field of fractions is $K$. By spreading out, there exists a morphism $f \colon \cX \to \cX'$ with $\cX$ also smooth, separated and of finite type over $k$, such that the generic fiber of $f$ is $X$. Note that $r = \dim(\cX)$ (the localization of a $p$-basis is a $p$-basis), so we obtain by \cite[Theorem 4.1]{Ekedahl_Duality_Hodge_Witt} that there is a natural isomorphism $W_n\Omega_{\cX}^r \cong W_n\omega_{\cX}$, and it is explained in \cite[Remark 2.2.2]{Baudin_Witt_GR_vanishing_and_applications} why this isomorphism preserves all relevant trace maps (i.e. $W_{n - 1}\omega_{\cX} \inj W_n\omega_{\cX}$ and $C^n_e \colon F^e_*W_n\omega_{\cX} \to W_n\omega_{\cX}$). Localizing gives the desired result by \cite[Proposition I.1.11]{Illusie_Complexe_de_de_Rham_Witt_et_cohomologie_cristalline}. The same results hold for $Y$.
	
	Note that the usual pullback of Witt differential forms map $W_n\omega_X \cong W_n\Omega_X^r \to f_*W_n\Omega^r_Y \cong f_*W_n\omega_Y$ (see \cite[Equation I.1.12.2]{Illusie_Complexe_de_de_Rham_Witt_et_cohomologie_cristalline}) preserves the Cartier operator and the inclusions $W_{n -1}\omega \inj W_n\omega$ by \cite[Lemma 5.1.9]{Baudin_Witt_GR_vanishing_and_applications} and localizing. Hence, we obtain by \autoref{def_quasi_stable_sections} that $f^* \colon H^0(X, \omega_X) \cong H^0(X, \Omega_X^r) \to H^0(Y, \Omega_Y^r) \cong H^0(Y, \omega_Y)$ sends each $q^eS^0_n(X, \omega_X)$ into $q^eS^0_n(Y, \omega_Y)$. Since $f$ is separable and hence generically étale, these maps are injective. 
\end{proof}
\begin{srem}
	The isomorphism $W_n\omega_X \cong W_n\Omega_X^r$ with all the compatibilities in the proof above was the only reason why we asked that $K$ is finitely generated over a perfect field in this paper, as opposed to any $F$-finite field. In particular, if analogues of the results of \cite{Bhatt_Blickle_Schwede_Tucker_F-finite_schemes_have_a_dc} existed in the Witt setup, we would be able to get rid of this unnatural (although probably harmless) assumption.
\end{srem}

\begin{sprop}\label{main_prop_quasi_stable_sections_pull_back}
	Let $f \colon Y \to X$ be a generically finite and surjective morphism between normal proper varieties, with $X$ regular. If $q^{\infty}S^0(X, \omega_X) \neq 0$, then $q^{\infty}S^0(Y, \omega_Y) \neq 0$. In particular, $X$ is not geometrically uniruled.
\end{sprop}
\begin{proof}
	Write \[ \begin{tikzcd}
		Y \arrow[r, "\pi"] & Z \arrow[r, "g"] & Z' \arrow[r, "h"] & X,
	\end{tikzcd} \] where $\pi$ is birational, $g$ is finite purely inseparable and $h$ is finite separable. Let $e > 0$ be an integer such that $F^e \colon Z' \to Z'$ factors through $Z$. Consider the diagram 
	\[ \begin{tikzcd}
		Y \arrow[r, "\pi"]              & Z \arrow[r, "g"]                                & Z' \arrow[r, "h"]    & X, \\
		Y' \arrow[u] \arrow[r, "\pi'"'] & Z' \arrow[ru, "F^e"'] \arrow[u] \arrow[r, "h"'] & X \arrow[ru, "F^e"'] &  
	\end{tikzcd} \] where we set $Y'$ to be the normalization of $(Y \times_Z Z')_{\red}$ (in particular, $\pi'$ is birational). Since $h \circ \pi' \colon Y' \to X$ is separable and $X$ is regular, we obtain by \autoref{pullback_by_separable_maps} that $q^{\infty}S^0(Y', \omega_{Y'}) \neq 0$. Since $K(Y)^{p^e} \inc K(Y')$, we deduce by \autoref{purely_inseparable_case} that $q^\infty S^0(Y, \omega_Y) \neq 0$. The final assertion follows readily, since $H^0(\bP^1, \omega_{\bP^1}) = 0$.
\end{proof}

\begin{srem}
	\begin{itemize}
		\item This result can be extended in a straightforward way to any normal proper variety $X$ admitting a resolution $\pi \colon X' \to X$ such that $\pi_*W_n\omega_{X'} = W_n\omega_X$ for all $n \geq 1$ (see the proof of \autoref{cor_main_thm_regular_case}). This latter assumption is for example satisfied whenever $X$ has klt singularities by \cite[Lemma 3.8]{Baudin_Kawakami_Roesler_On_GR_for_klt_CM_schemes}, or quasi-$F$-rational singularities by the same argument as in the proof of  \cite[Theorem 5.2.4]{Baudin_Witt_GR_vanishing_and_applications} (see also \cite[Corollary 3.15]{Quasi_F_splittings_III}), as long as it admits a resolution.
		\item Note that this result shows more generally that if the inseparability part of $f \colon Y \to X$ (i.e. the map $g$ in the proof) factors through $F^e$ and $q^eS^0_n(X, \omega_X) \neq 0$ for some $n \geq 1$, then $H^0(Y, \omega_Y) \neq 0$. In particular, the quotient $Y$ of a regular quasi-$F$-split $K$-trivial variety by a $1$-foliation must satisfy $H^0(Y, \omega_Y) \neq 0$.
	\end{itemize}
\end{srem}

\begin{scor}[\autoref{main_thm_qFsplit_non_uniruled}]\label{cor_main_thm_regular_case}
	Let $X$ be a normal, proper and quasi-$F^{\infty}$-split variety with klt singularities such that $K_X \equiv 0$ and $H^0(X, \cO_X) = K$. Assume furthermore that the cyclic cover trivializing $K_X$ admits a resolution of singularities. Then $X$ has geometrically klt singularities. If $K_X$ is in addition Cartier, then $X$ is not geometrically uniruled. 
\end{scor}
\begin{srem}
	Recall from the proof of \autoref{qFsplit_implies_geom_normal} that $K_X$ is torsion of prime-to-$p$ order, so the associated finite cover is quasi-étale. Thus, if $X$ is regular to begin with, then this cover is étale and there is no further assumption about resolution of singularities.
\end{srem}
\begin{proof}
	Throughout, we will use the notations from the proof of \autoref{qFsplit_implies_geom_normal} (note that our argument in this proof and \cite[Proposition 3.24]{Quasi_F^e_splittings_and_quasi-F_regularity} show that $q^{\infty}S^0(Y, \omega_Y) \neq 0$). Let $\pi \colon Y' \to Y$ be a resolution of singularities, and let us show that $Y$ has geometrically canonical singularities and is not geometrically uniruled (without the assumption that $K_X$ is Cartier). Given that $Y$ has canonical singularities, we know by \cite[Lemma 3.8]{Baudin_Kawakami_Roesler_On_GR_for_klt_CM_schemes} that $\pi_*W_n\omega_{Y'} = W_n\omega_Y$ for all $n \geq 1$, so it follows from \autoref{def_quasi_stable_sections} that $q^{\infty}S^0(Y', \omega_{Y'}) \neq 0$. We then deduce from \autoref{main_prop_quasi_stable_sections_pull_back} that $Y'$ (and hence $Y$) is geometrically normal and not geometrically uniruled. 
	
	Assume now by contradiction that the singularities of $Y$ are not geometrically canonical. Then there exists $e > 0$ and a projective birational morphism $\theta \colon Z \to Y_{k^{1/p^e}}$ from a normal variety $Z$ such that the inclusion $\theta_*\omega_Z \to \omega_{Y_{k^{1/p^e}}}$ is strict, since $K_{Y_{k^{1/p^e}}}$ is Cartier (we are using \autoref{BPRZ}). Given that $\omega_{Y_{k^{1/p^e}}} \cong \cO_{Y_{k^{1/p^e}}}$, we deduce that $H^0(Z, \omega_Z) = 0$. Let $Z' \to Z$ be a further blowup with $Z'$ normal such that the induced map $Z' \to Y$ factors through $Y'$. Then $H^0(Z', \omega_{Z'}) \inc H^0(Z, \omega_Z) = 0$, which contradicts \autoref{main_prop_quasi_stable_sections_pull_back} since $q^{\infty}S^0(Y', \omega_{Y'}) \neq 0$. Thus, we have proven that $Y$ has geometrically canonical singularities, and is not geometrically uniruled.
	
	By \cite[Proposition 2.50]{Kollar_Singularities_of_the_MMP}, we automatically deduce that $X$ has geometrically klt singularities. If in addition $K_X$ is Cartier, then the map $f \colon Y \to X$ is étale, so we obtain that also $X$ is not geometrically uniruled. \qedhere

\end{proof}
%
%Note that the same proof also shows the klt version:
%
%\begin{scor}\label{cor_main_thm_singular_version}
%	Let $X$ be a normal, proper and quasi-$F^e$-split variety with klt singularities such that $K_X \equiv 0$. Then $X$ is geometrically normal. Furthermore, if the cyclic cover trivializing $K_X$ admits a resolution of singularities, then also $X_{k^{1/p^e}}$ has klt singularities. In particular, $X$ is geometrically klt if it is quasi-$F^{\infty}$-split.
%\end{scor}
%\begin{proof}
%	As in the proof of \autoref{cor_main_thm_regular_case}, we know that $K_X$ is $\bZ_{(p)}$-torsion, so let $f \colon Y \to X$ be the associated quasi-étale cover. To verify that $X$ is geometrically normal, it suffices to show that that it is geometrically $R_1$, which can be checked at codimension one points by definition. Hence, it is enough to show that $Y$ is geometrically normal. Since $Y$ is quasi-$F^e$-split in codimension $1$ by \cite[Proposition 3.24]{Quasi_F^e_splittings_and_quasi-F_regularity} \JB{bababaaaa}
%\end{proof}

\section{A $K$-trivial surface which is quasi-$F$-split but not quasi-$F^{\infty}$-split}

Our objective in this section is to prove \autoref{intro_example}. Even though all the results beforehand will not be necessary, we include them here because we believe that they give a good idea of the general picture.

\begin{slem}\label{partial_positive_answer}
	Let $X$ be a $d$-dimensional normal proper variety, with $\omega_X \cong \cO_X$. Then the following holds:
	\begin{enumerate}
		\item\label{first_item} if the natural surjection $H^d(X, W\cO_X) \to H^d(X, \cO_X)$ is an isomorphism, then either the Frobenius action on $H^d(X, \cO_X)$ is non-zero (equivalently $X$ is globally $F$-split), or $X$ is not quasi-$F$-split;
		\item\label{second_item} if $H^d(X, W\cO_X) \otimes \bQ \neq 0$, then $X$ is quasi-$F^{\infty}$-split;
		\item\label{third_item} if $X$ is quasi-$F^{\infty}$-split, not globally $F$-split and $H^0(X, \cO_X)$ is a perfect field, then $H^d(X, W\cO_X) \otimes \bQ \neq 0$.
	\end{enumerate}
\end{slem}
\begin{proof}
	We first prove \autoref{first_item}, so let us assume that $X$ is quasi-$F$-split and let us show that the Frobenius action on $H^d(X, \cO_X)$ is non-zero. Since $X$ is quasi-$F$-split, we deduce by \autoref{quasi-stable_maps_into_quasi-stable} and Serre duality that for some $n > 0$, the map $H^d(X, W_n\cO_X) \to H^d(X, F_*W_n\cO_X)$ is non-zero. Note that the assumption forces that $H^d(X, W_n\cO_X) \to H^d(X, \cO_X)$ is an isomorphism, so we deduce the result. Let us now prove \autoref{second_item}, so fix $e > 0$. Since the action of $p^e$ on $H^d(X, W\cO_X)$ is non-zero by assumption and the fact that $p = FV$, we deduce that $F^e \colon H^d(X, W\cO_X) \to H^d(X, F^e_*W\cO_X)$ is also non-zero. Given that $H^d(X, W\cO_X) = \lim H^d(X, W_n\cO_X)$ by the Mittag--Leffler condition, we deduce that there exists $n \geq 1$ such that $F^e \colon H^d(X, W_n\cO_X) \to H^d(X, F^e_*W_n\cO_X)$ is non-zero. We therefore deduce by \autoref{quasi-stable_maps_into_quasi-stable} that $X$ is $n$-quasi-$F^e$-split. Finally, let us prove \autoref{third_item}. Since $X$ is in particular quasi-$F$-split and is defined over a perfect field $k_0$, we deduce by \cite[Theorem 2.7.(2)]{Yobuko_Quasi_F_split_and_Hodge_Witt} (see also \cite{Nakkajima_Artin_Mazur_heights_and_Yobuko_heights}) that $M \coloneqq H^d(X, W\cO_X)$ is a finitely generated $W(k_0)$-module. Assume by contradiction that $p^sM = 0$ for some $s > 0$. Then $M$ has finite length, so the inverse system $\{H^d(X, W_n\cO_X)\}_{n \geq 1}$ is eventually constant. By the same argument as above, $F^e \colon M \to F^e_*M$ is not zero for any $e > 0$, so we deduce from the above that there exists $n > 0$ such that $F^e \colon H^d(X, W_n\cO_X) \to H^d(X, F^e_*W_n\cO_X)$ is non-zero for any $e > 0$. Using the exact sequences \[ \begin{tikzcd}
		{H^d(X, F_*W_{n - 1}\cO_X)} \arrow[r, "V"] & {H^d(X, W_n\cO_X)} \arrow[r, "R^{n - 1}"] & {H^d(X, \cO_X)} \arrow[r] & 0,
	\end{tikzcd} \] we deduce that the same property holds on either $H^d(X, W_{n - 1}\cO_X)$ or $H^d(X, \cO_X)$. We can keep going inductively to deduce that the Frobenius action on $H^d(X, \cO_X)$ is non-zero, i.e. $X$ is globally $F$-split. 
\end{proof}

\begin{scor}\label{no_issue_if_bocktein_zero}
	Let $X$ be a normal proper variety such that $\omega_X \cong \cO_X$ and $H^0(X, \cO_X) = k_0$ is a perfect field. If all the sequences \[ \begin{tikzcd}
		{H^d(X, F_*W_{n - 1}\cO_X)} \arrow[r, "V"] & {H^d(X, W_n\cO_X)} \arrow[r, "R^{n - 1}"] & {H^d(X, \cO_X)} \arrow[r] & 0,
	\end{tikzcd} \] are also exact on the left (e.g. if $H^{d - 1}(X, \cO_X) = 0$) and $X$ is quasi-$F$-split, then it is quasi-$F^{\infty}$-split.
\end{scor}
\begin{proof}
	Since $X$ is quasi-$F$-split, we know by \cite[Theorem 2.7.(2)]{Yobuko_Quasi_F_split_and_Hodge_Witt} that $H^d(X, W\cO_X)$ is finitely generated over $W(k_0)$. Given that $H^d(X, \cO_X) \cong k_0$, we deduce by the assumption that $\length_{W(k_0)}(H^d(X, W_n\cO_X)) = n$ for all $n \geq 1$, so $H^d(X, W\cO_X)$ cannot have finite length. We conclude the proof by \autoref{partial_positive_answer}.\autoref{second_item}.
\end{proof}

\begin{slem}\label{crystalline_cohomology_criterion}
	Let $X$ be a $d$-dimensional quasi-$F$-split smooth proper variety such that $\omega_X \cong \cO_X$ and $H^0(X, \cO_X) = k_0$ is a perfect field. If both groups $H^d_{\crys}(X/W)$ and $H^{d + 1}_{\crys}(X/W)$ are $p$-torsion-free, then $H^d(X, W\cO_X) \otimes \bQ \neq 0$, and hence $X$ is quasi-$F^{\infty}$-split.
\end{slem}
\begin{proof}
	The last conclusion follows from \autoref{partial_positive_answer}.\autoref{second_item}, so we only need to show that $H^d(X, W\cO_X) \otimes \bQ = 0$. Throughout, we will use the notations of \cite{Ekedahl_Diagonal_complexes_and_F_gauge_structures}, which we partially review now. Fix $i, j \geq 0$. We set $h^{i, j} \coloneqq \dim H^j(X, \Omega_X^i)$. Given $\lambda \in \bQ_{\geq 0}$, we set $h^{i + j}_{\crys, \lambda}$ to be the dimension of the slope $\lambda$ part of the $F$-isocrystal $H^{i + j}_{\crys}(X/W) \otimes \bQ$, and we set \[ m^{i, j} \coloneqq \sum_{\lambda \in [i, i + 1[}(i + 1 - \lambda)h^{i + j}_{\crys, \lambda} + \sum_{\lambda \in [i - 1, i[} (\lambda - i + 1) h^{i + j}_{\crys, \lambda}. \] We also let $T^{i, j}$ the number defined at the beginning of \cite[Section 0.6]{Ekedahl_Diagonal_complexes_and_F_gauge_structures} (whose purpose is to compute how much the slope spectral sequence fails to degenerate integrally) and we set \[ h^{i, j}_W \coloneqq m^{i, j} + T^{i, j} -2T^{i - 1, j + 1} + T^{i - 2, j + 2}. \] In particular, we have \[ h^{0, d}_W = \sum_{\lambda \in [0, 1[}(1 - \lambda)h^d_{\crys, \lambda} + T^{0, d}. \] Given that $H^d(X, W\cO_X)$ is finitely generated by the quasi-$F$-split assumption and \cite[Theorem 2.7.(2)]{Yobuko_Quasi_F_split_and_Hodge_Witt}, we deduce by the definition of $T^{0, d}$ and \cite[Corollaire II.3.9]{Illusie_Raynaud_Les_suites_spectrales_associees_au_cox_de_de_Rham_Witt} that $T^{0, d} = 0$ (in fact, $T^{0, j} = 0$ for all $j$). If we further assume by contradiction that $H^d(X, W\cO_X) \otimes \bQ = 0$, then it means by \cite[Corollaire II.3.5]{Illusie_Complexe_de_de_Rham_Witt_et_cohomologie_cristalline} that $H^d_{\crys}(X/W) \otimes \bQ$ has no slope of degree $< 1$, so we have that $h^{0, d}_W = 0$. 
	
	We will now show that $h^{0, d}_W = 1$, giving a contradiction. Given that $H \coloneqq \bigoplus_s H^s_{\crys}(X/W)$ is a coherent module over the Raynaud ring \cite[Théorème II.2.2]{Illusie_Raynaud_Les_suites_spectrales_associees_au_cox_de_de_Rham_Witt} and \[ H^d_{\crys}(X/W)/p \expl{\cong}{by the $p$-torsion-freeness assumptions and \cite[Equation II.4.9.1]{Illusie_Complexe_de_de_Rham_Witt_et_cohomologie_cristalline}} H^d_{\dR}(X/k) \cong \bigoplus_{a + b = d} H^b(X, \Omega_X^a), \] (the second isomorphism follows from the quasi-$F$-split assumption by \cite[Theorem 1.1]{Petrov_Decomposition_of_the_dR_cpx_for_qFsplit_varieties}), it follows that $H$ is a Mazur-Ogus object in degree $d$ (see \cite[Definition IV.1.1]{Ekedahl_Diagonal_complexes_and_F_gauge_structures}). By \cite[Corollary IV.3.3.1]{Ekedahl_Diagonal_complexes_and_F_gauge_structures}, we deduce in particular that \[ h^{0, d}_W = h^{0, d} = 1. \qedhere \]
\end{proof}

Here is our main result in this section, answering negatively the question raised in \cite{Quasi_F_splittings_III} below their Theorem H.

\begin{sthm}[{\autoref{intro_example}}]
	Let $S$ be a quasi-$F$-split, smooth projective surface over an algebraically closed field $k$ such that $\omega_S \cong \cO_S$. Then $S$ is quasi-$F^{\infty}$-split, unless $p = 2$ and $S$ is given by the diagonal quotient \[ (E_1 \times E_2)/\bZ/2\bZ, \] where: 
	\begin{enumerate}
		\item $E_1$ is an ordinary elliptic curve acted on by its unique non-zero point of $2$-torsion;
		\item $E_2$ is a supersingular elliptic curve acted on by the sign involution.
	\end{enumerate} 
	In particular, there exists a log canonical threefold singularity $X$ with $K_X$ Cartier that is quasi-$F$-pure but not quasi-$F^{\infty}$-pure.
\end{sthm}
\begin{proof}
	If $S$ is either a K3 surface of an abelian surface, then $S$ is automatically quasi-$F^{\infty}$-split by either \autoref{no_issue_if_bocktein_zero} (see also \cite[Théorème 2]{Serre_Quelques_proprietes_des_varietes_abeliennes_en_car_p} in the case of an abelian surface), or by \autoref{crystalline_cohomology_criterion} and \cite[Sections II.7.1 and II.7.2]{Illusie_Complexe_de_de_Rham_Witt_et_cohomologie_cristalline}, or by \cite[Corollaries 7.2 and 7.3]{Quasi_F^e_splittings_and_quasi-F_regularity}. If $S$ is an Enriques surface, then automatically $p = 2$ and have that $\length(H^2(X, W\cO_S)) = 1$ by \cite[Proposition 7.3.2]{Illusie_Complexe_de_de_Rham_Witt_et_cohomologie_cristalline}, so $S$ is in fact globally $F$-split by \autoref{partial_positive_answer}.\autoref{first_item} (and hence quasi-$F^{\infty}$-split). Note that $S$ cannot be quasi-bielliptic by \autoref{intro_qF_split_implies_geom_normal}. Finally, assume that $S$ is a bielliptic surface, and hence $p \in \{2, 3\}$. By \cite[Proposition E.4.3 and the paragraph after its proof]{Lang_Quasi_elliptic_surfaces_in_characteristic_three}, we have that all bielliptic surfaces with trivial canonical bundle apart from the one of the theorem satisfy that $\length(H^2(S, W\cO_S)) = 1$, so we conclude as in the case of Enriques surfaces. We are therefore left to verify that the bielliptic surface from the theorem is quasi-$F$-split but not quasi-$F^2$-split. First, note that $S$ is quasi-$F$-split but not globally $F$-split by \cite[Theorem 4.25]{Kawakami_Takamatsu_Yoshikawa_Fedder_type_criteria_for_quasi_F_splittings_II}. We also know by \cite[Paragraph after the proof of Proposition 4.3]{Lang_Quasi_elliptic_surfaces_in_characteristic_three} that the first Bockstein operator of $S$ is zero, and that $H^2(S, W\cO_S)$ has length $2$. This fact on the first Bockstein operator means that the sequence \[ \begin{tikzcd}
		0 \arrow[r]  & {H^2(S, F_*\cO_S)} \arrow[rr, "V"] &  & {H^2(S, W_2\cO_S)} \arrow[rr, "R"] &  & {H^2(S, \cO_S)} \arrow[r] & 0
	\end{tikzcd} \] is exact. Since $S$ is not globally $F$-split (its abelian cover contains a supersingular elliptic curve), the Frobenius action on the left and right term are zero, so the Frobenius action on the middle term is nilpotent of order $2$. On the other hand, since $H^2(S, W\cO_S) \to H^2(S, W_2\cO_S)$ is a surjection of modules of length $2$, it must be an isomorphism. We then conclude by \autoref{quasi-stable_maps_into_quasi-stable} and Serre duality that $S$ is not quasi-$F^2$-split. The statement after ``In particular'' follows from taking cones by \cite[Theorem 6.5]{Quasi_F_splittings_III}.
\end{proof}

\section{Quasi-$F$-purity and log canonicity}

Our objective here is to prove \autoref{intro_qFsplit_implies_lc}. The following analogue of \cite[Theorem 3.2]{Baudin_Kawakami_Roesler_On_GR_for_klt_CM_schemes} (see also \cite[Theorem 3.7]{Kawakami_Local_vanishing_for_F_pure_threefolds}) will be essential to our approach.

\begin{sprop}\label{Tatsuro_vanishing}
	Let $X$ be a normal variety, and let $\Delta \geq 0$ be a divisor such that $K_X + \Delta$ is $\bQ$--Cartier. Let $\pi \colon Y \to X$ be a projective birational morphism with $Y$ regular, and write \[ K_Y + \pi^{-1}_*\Delta \sim_{\bQ} \pi^*(K_X + \Delta) + E_+ - E_-, \] with $E_+$ and $E_-$ effective and $\pi$-exceptional. Then for all $j \geq 0$, \[ \codim \Supp(R^i\pi_*\cO_Y(-\lfloor E_- \rfloor)) > i + 1. \]
\end{sprop}
\begin{proof}
	The proof will be almost identical to that of \cite[Theorem 3.2]{Baudin_Kawakami_Roesler_On_GR_for_klt_CM_schemes}. Let $d \coloneqq \dim(X)$, and let us show the result by induction on $d$. It follows by localizing at points of $X$ and using the induction hypothesis that it is enough to show that \[ R^{d - 1}\pi_*\cO_Y(-\lfloor E_- \rfloor) = 0.\] 
	By \cite[Theorem II.7.17]{Hartshorne_Algebraic_Geometry}, there exists a closed subscheme $Z\subseteq X$ such that $\pi$ is the blow-up of $X$ along $Z$. Let us denote $H \coloneqq \pi^{-1}(Z)$, so that $\cO_{Y}(-H)$ is $\pi$-ample by \stacksprojs{02NS}{02OS}. We can take $n\gg 0$ such that $R^{d-1}\pi_*\cO_Y(-nH - \lfloor E_- \rfloor)=0$ by relative Serre vanishing. Let us write
	\begin{align*}
		nH= \sum_{i\in I}r_iE_i \: + \: G,
	\end{align*}
	for some positive integers $r_i>0$, where the $E_i$'s are exactly the $\pi$-exceptional irreducible divisors. Consider the short exact sequence
	\begin{align*}
		0\to\cO_Y(-nH - \lfloor E_- \rfloor)\to\cO_Y(-\sum_{i \in I}r_i E_i  -  \lfloor E_- \rfloor)\to\cO_G(-\sum_ir_iE_i  -  \lfloor E_- \rfloor)\to 0.
	\end{align*}
	As $R^{d-1}\pi_*\cO_Y(-nH - \lfloor E_- \rfloor)=0$ and $R^{d-1}\pi_*\cO_G(-\sum_i r_iE_i - \lfloor E_- \rfloor)=0$ (the fibers of $G \to \pi(G)$ have dimension $\leq d - 2$), we deduce that $R^{d - 1}\pi_*\cO_Y(-\sum_i r_iE_i - \lfloor E_- \rfloor) = 0$. To conclude the proof, we are then left to show the following:
	\begin{claim*}
		If $R^{d-1}\pi_{*}\cO_Y(-\sum_{i\in I} n_iE_i - \lfloor E_- \rfloor)=0$ for some $(n_i)_{i\in I}\in\bZ_{\geq 0}$ satisfying $\sum_{i\in I} n_i\geq 1$, then there exists $j\in I$ such that $n_j\geq 1$ and \[R^{d-1}\pi_{*}\cO_Y\left(-(n_j-1)E_j-\sum_{i\in I\setminus\{j\}} n_iE_i - \lfloor E_- \rfloor\right)=0.\] 
	\end{claim*}
	\noindent \textit{Proof of the claim.}
	For now, fix $j \in I$ with $n_j \geq 1$. We will pick a specific $j$ later. Consider the short exact sequence
	\begin{multline*}
		0\to \cO_Y\left(-\sum_{i\in I} n_iE_i - \lfloor E_- \rfloor\right) \to  \cO_Y\left(-(n_j-1)E_j-\sum_{i\in I\setminus\{j\}} n_iE_i - \lfloor E_- \rfloor\right)\\\to \cO_{E_j}\left(E_j-\sum_{i\in I} n_iE_i - \lfloor E_- \rfloor\right)\to 0. 
	\end{multline*}
	We aim to show that \[R^{d-1}\pi_{*}\cO_Y\left(-(n_j-1)E_j-\sum_{i\in I\setminus\{j\}} n_iE_i - \lfloor E_- \rfloor\right)=0,\] which is equivalent to  \[R^{d-1}\pi_{*}\cO_{E_j}\left(E_j-\sum_{i\in I} n_iE_i  - \lfloor E_- \rfloor\right)=0.\]
	If $\dim \pi(E_j)>0$, then this is immediate since then the fibers of $E_j \to \pi(E_j)$ have dimension $\leq d - 2$. If $\dim \pi(E_j)=0$, then
	\begin{align*}
		R^{d-1}\pi_{*}\cO_{E_j}\left(E_j-\sum_{i\in I} n_iE_i - \lfloor E_- \rfloor\right)&=H^{d-1}\left(E_j, \cO_{E_j}\left(E_j-\sum_{i\in I} n_iE_i - \lfloor E_- \rfloor\right)\right)\\
		&\cong H^0\left(E_j, \cO_{E_j}\left(K_{E_j}-E_j+\sum_{i\in I} n_iE_i + \lfloor E_- \rfloor\right)\right)^{\vee} \\
		&\cong H^0\left(E_j, \cO_{E_j}\left(K_{Y}+\sum_{i\in I} n_iE_i + \lfloor E_- \rfloor\right)\right)^{\vee}.
	\end{align*}
	To conclude that the latter group vanishes, it is then enough to show that $-(K_{Y}+\sum_{i\in I} n_iE_i + \lfloor E_- \rfloor)|_{E_j}$ is $\pi|_{E_j}$-big. Let us now find some $j \in I$ achieving this. Write \[ K_Y + \pi^{-1}_*\Delta + \sum_{i \in I}a_iE_i +  \lfloor E_- \rfloor \sim_{\bQ} \pi^*(K_X + \Delta), \] so that $a_i \in \bQ_{< 1}$ for all $i \in I$ by construction (note that $\bigcup_{i \in I}E_i = \Supp(\sum_i r_iE_i)$ coincides with the codimension one part of exceptional locus). Let $J\coloneqq \{i\in I \mid n_i-a_i> 0\}\subset I$.
	Note that $J\neq \emptyset$ since $\sum_{i\in I} n_i\geq 1$, while $a_i<1$ for all $i \in I$. Set
	\begin{align*}
		t\coloneqq\underset{i\in J}{\max}\left\{\frac{n_i-a_i}{r_i}\right\}\in\bQ_{>0},
	\end{align*}
	and let $j\in J$ be an index where the maximum is attained. By construction, we have that
	\[E_j\not\subset \Supp\left(t\left(\sum_{i\in I} r_iE_i\right)-\sum_{i\in J} (n_i-a_i)E_i\right),\]
	so 
	\[
	-\left.\left(\sum_{i\in J} (n_i-a_i )E_i\right)\right|_{E_j}=\left.\left(-tG-t\sum_{i\in I} r_iE_i\right)\right|_{E_j}+\left.\left(tG+t\sum_{i\in I} r_iE_i-\sum_{i\in J} (n_i-a_i)E_i\right)\right|_{E_j}
	\]
	is $\pi|_{E_j}$-big (recall that $-G-\sum_{i\in I} r_iE_i$ is $\pi$-ample), whence
	\begin{align*}
		-\left.\left(K_{Y}+\sum_{i\in I} n_iE_i + \lfloor E_- \rfloor \right)\right|_{E_j} & \sim_{\bQ,\pi|_{E_j}} -\left.\left(\sum_{i\in I} (n_i-a_i)E_i\right)\right|_{E_j}+\pi^{-1}_{*}\Delta|_{E_j}\\
		&= -\left.\left(\sum_{i\in J} (n_i-a_i )E_i\right)\right|_{E_j}+\left.\left(\sum_{i\in I\setminus J} (a_i-n_i )E_i\right)\right|_{E_j}+\pi^{-1}_{*}\Delta|_{E_j}
	\end{align*}
	is also $\pi|_{E_j}$-big.
\end{proof}

\begin{scor}[{\cite[Lemma 3.8]{Baudin_Kawakami_Roesler_On_GR_for_klt_CM_schemes}}]\label{pushforward_of_witt_is_what_we_think}
	Notations as in \autoref{Tatsuro_vanishing}. Then for all $n \geq 1$, we have \[ \pi_*W_n\omega_{(Y, \lfloor E_- \rfloor)} = W_n\omega_X. \]
\end{scor}
\begin{proof}
	Consider the exact triangle \[ \begin{tikzcd}
		\pi_*W_n\cI_{\lfloor E_- \rfloor} \arrow[rr] &  & R\pi_*W_n\cI_{\lfloor E_- \rfloor} \arrow[rr] &  & \tau_{\geq 1}(R\pi_*W_n\cI_{\lfloor E_- \rfloor}) \arrow[rr, "+1"] &  & {}
	\end{tikzcd} \] By \autoref{Tatsuro_vanishing} and the usual exact sequences \[ \begin{tikzcd}
	0 \arrow[r]   & F_*W_{n - 1}\cI_{\lfloor E_- \rfloor} \arrow[r, "V"] & W_n\cI_{\lfloor E_- \rfloor} \arrow[r, "R^{n - 1}"] & \cO_Y(-\lfloor E_- \rfloor) \arrow[r] & 0,
	\end{tikzcd} \] we obtain that $\codim \Supp(R^i\pi_*W_n\cI_{\lfloor E_- \rfloor}) > i + 1$ for all $i$. Hence, if we set $\bD(-) \coloneqq \cR\HHom(-, W_n\omega_X^{\bullet})$, we obtain by the same argument as in the proof of \cite[Lemma 3.4]{Baudin_Kawakami_Roesler_On_GR_for_klt_CM_schemes} that \[ \cH^{-(d - 1)}(\bD(\tau_{\geq 1}(R\pi_*W_n\cI_{\lfloor E_- \rfloor}))) = 0. \] Thus, we obtain that 
	\begin{align*} 
		\HHom(\pi_*W_n\cI_{\lfloor E_- \rfloor}, W_n\omega_X) &= \cH^{-d}(\bD(\pi_*W_n\cI_{\lfloor E_- \rfloor})) \\
		& = \cH^{-d}(\bD(R\pi_*W_n\cI_{\lfloor E_- \rfloor})) \\
		&= \cH^{-d}(R\pi_*W_n\omega_{(Y, \lfloor E_- \rfloor)}^{\bullet}) \\
		& = \pi_*W_n\omega_{(Y, \lfloor E_- \rfloor)}.
	\end{align*}
	Finally, note that the natural map $W_n\omega_X \to \HHom(\pi_*W_n\cI_{\lfloor E_- \rfloor}, W_n\omega_X)$ is an isomorphism on a big open subset. Since $W_n\omega_X$ is $S_2$ (\stacksproj{0AWE}) and the restriction map of sections on $\HHom(\pi_*W_n\cI_{\lfloor E_- \rfloor}, W_n\omega_X)$ is injective, this map is an isomorphism.
\end{proof}

Before stating our theorem about log canonicity, we will also need the following:

\begin{slem}\label{pushout_sucks_but_pullback_vibes}
	Let $0 \neq \cI \inc \cJ$ be the inclusion of quasi-coherent ideals on a variety $X$, Then the commutative diagram \[ \begin{tikzcd}
		{\HHom(W_n\cJ, W_n\omega_X)} \arrow[r]   & {\HHom(W_n\cI, W_n\omega_X)}     \\
		{\HHom(\cJ, W_n\omega_X)} \arrow[u] \arrow[r]  & {\HHom(\cI, W_n\omega_X)} \arrow[u]
	\end{tikzcd}  \] is a pullback square. Furthermore, the natural maps $\HHom(\cJ, \omega_X) \to \HHom(\cJ, W_n\omega_X)$ and $\HHom(\cI, \omega_X) \to \HHom(\cI, W_n\omega_X)$ are isomorphisms.
\end{slem}
\begin{proof}
	Let $i \colon X \inj W_nX$ denote the natural map. Since $i^!W_n\omega_X^{\bullet} = \omega_X^{\bullet}$ by construction, we obtain that the natural injection $\omega_X \to i^{\flat}W_n\omega_X$ is an isomorphism (we set $i^{\flat}$ to be the non-derived right adjoint of $i_*$, so that $Ri^{\flat} = i^!$). Thus, the statements after ``Furthermore'' follow immediately by adjunction, so it is enough to show the pullback square property. Consider the commutative diagram \[ \begin{tikzcd}
		W_n\cI \arrow[d, "R^{n - 1}"', two heads] \arrow[r, "\inc"]       & W_n\cJ \arrow[d, two heads] \arrow[rdd, "R^{n - 1}", bend left]     &     \\
		\cI \arrow[r, "\theta"] \arrow[rrd, "\inc"', bend right=20] & W_n\cJ/(V(F_*W_{n - 1}\cI)) \arrow[rd, , "\phi", two heads] &     \\
		&                                                  & \cJ,
	\end{tikzcd} \] where $\theta$ is given by the composition $\cI \cong W_n\cI/V(F_*W_{n - 1}\cI) \inc W_n\cJ/V(F_*W_{n - 1}\cI)$, making the middle square a pushout square. Since the functor $\HHom(-, W_n\omega_X)$ transforms pullback squares into pushout squares, we are left to show that $\HHom(\ker(\phi), W_n\omega_X) = 0$. Note that \[ \ker(\phi) \cong \quot{F_*W_{n - 1}\cJ}{F_*W_{n - 1}\cI}  \] is supported on a subscheme of the form $W_nZ$, with $i_Z \colon Z \inj X$ strict. Hence, we have that $\HHom(\ker(\phi), W_n\omega_X) = \HHom(\ker(\phi), i_Z^{\flat}W_n\omega_X)$. Since $\dim(Z) < \dim(X)$ (recall that $X$ is integral), $W_n\omega_Z^{\bullet} = i_Z^!W_n\omega_X^{\bullet}$ is supported in degrees strictly bigger than those of $W_n\omega_X^{\bullet}$ by \stacksproj{0A7U}, so $i_Z^{\flat}W_n\omega_X = 0$. 
\end{proof}

Let $X$ be a normal variety, let $x \in X$ and assume that $K_X$ is $\bQ$-Cartier at $x$. Let $m > 0$ be the Cartier index of $K_X$, and write $m = p^er$ with $e \geq 0$ and $r > 0$ coprime to $p$. Up to shrinking $X$ around $x$, we may assume that $\cO_X(rp^eK_X) \cong \cO_X$, so we can set $X'$ to be the normalization of $\Spec_X(\bigoplus_{i = 1}^{r - 1}\cO_X(-ip^eK_X))$. We will call the induced quasi-étale $\mu_r$-cover $X' \to X$ a \emph{prime-to-$p$ index cover of $X$ at $x$}.

\begin{sthm}[{\autoref{intro_qFsplit_implies_lc}}]\label{qF_pure_implies_lc_improved}
	Let $X$ be a $n$-quasi-$F^e$-pure, normal variety such that there exists $r > 0$ coprime to $p$ with $rp^eK_X$ is Cartier. If a prime-to-$p$ index cover $X$ at $x \in X$ admits a log resolution, then $X$ has log canonical singularities at $x$.
\end{sthm}
\begin{proof}
	Let $f \colon X' \to X$ be a prime-to-$p$ cover of $X$ at $x$ admitting a log resolution. By construction, $p^eK_{X'}$ is Cartier. Furthermore, since $f$ is quasi-étale, we know by \cite[Proposition 3.24]{Quasi_F^e_splittings_and_quasi-F_regularity} that $X'$ is quasi-$F^e$-pure on a big open subset, and hence quasi-$F^e$-pure by the $S_2$ property. If we could show that $X'$ has log canonical singularities, then the same argument as in the proof of \cite[Proposition 3.5]{Braun_Greb_Langlois_Moraga_Reductive_quotients_of_klt_singularities} would show that $X$ also has log canonical singularities. Thus, we may replace $X'$ by $X$ and assume that $p^eK_X$ is Cartier and admits a log resolution $\pi \colon Y \to X$. It is therefore enough to show that the discrepancies of the exceptional divisors of $Y$ are at least $-1$. Write \[ p^eK_Y + p^eE_- \sim \pi^*(p^eK_X) + p^eE_+, \] and assume by contradiction that $E_-$ admits a component with coefficient strictly bigger than $1$. By \cite[Corollary 2.31]{Kollar_Mori_Birational_geometry_of_algebraic_varieties}, we can make this coefficient as big as needed, so we may assume that $\lfloor E_- \rfloor$ is not a reduced divisor. In particular, the morphism $\cI_{\lfloor E_- \rfloor} \to F^e_*\cI_{\lfloor E_- \rfloor}$ factors through some $\cI_{E'}$ for some $\bZ$-divisor $E' \inc \lfloor E_- \rfloor$ that is strictly smaller than $\lfloor E_- \rfloor$. The same factorization therefore holds at the level of $W_n$, so applying $\HHom(-, W_n\omega_Y)$ and twisting by $W_n\cO_Y((p^e - 1)K_Y + p^eE_- - \lfloor E_- \rfloor)$ gives a commutative diagram \[ \begin{tikzcd}
		{F^e_*W_n\omega_{(Y, \lfloor E_- \rfloor)}(p^e(p^e - 1)K_Y + p^e(p^eE_- - \lfloor E_- \rfloor))} \arrow[d] &  & \cP_2 \arrow[ll, hook'] \arrow[d]           \\
		{W_n\omega_{(Y, E')}((p^e - 1)K_Y + p^eE_- - \lfloor E_- \rfloor)} \arrow[d]                               &  & \cP_1 \arrow[ll, hook'] \arrow[d]           \\
		{W_n\omega_{(Y, \lfloor E_- \rfloor)}((p^e - 1)K_Y + p^eE_- - \lfloor E_- \rfloor)}                        &  & \cO_Y(p^eK_Y + p^eE_-), \arrow[ll, hook']
	\end{tikzcd} \] where $\cP_1$ and $\cP_2$ are defined so that both squares are pullback squares. We will pushforward this down to $X$ to obtain a contradiction with $n$-quasi-$F^e$-purity (which we will only use at the end of the proof), so let us identify each relevant piece. Note that 
	\begin{align*} 
		p^e(p^e - 1)K_Y + p^e(p^eE_- - \lfloor E_- \rfloor) & \sim \pi^*(p^e(p^e - 1)K_X) + p^e(p^e - 1)E_+ + p^e(E_- - \lfloor E_- \rfloor),
	\end{align*} where the part after $\pi^*(p^e(p^e - 1)K_X)$ above is effective and $\pi$-exceptional, so there exists a natural map \[ F^e_*W_n\omega_{(Y,\lfloor E_- \rfloor)}(\pi^*(p^e(p^e - 1)K_X)) \to F^e_*W_n\omega_{(Y, \lfloor E_- \rfloor)}(p^e(p^e - 1)K_Y + p^e(p^eE_- - \lfloor E_- \rfloor)). \] Pushing this down and using \autoref{pushforward_of_witt_is_what_we_think} gives a natural map \[ F^e_*W_n\omega_X(p^e(p^e - 1)K_X) \to F^e_*\pi_*W_n\omega_{(Y, \lfloor E_- \rfloor)}(p^e(p^e - 1)K_Y + p^e(p^eE_- - \lfloor E_- \rfloor)). \] Since the first sheaf is $S_2$ (\stacksproj{0AWE}), both sheaves are torsion-free and this map is an isomorphism on a big open subset, we deduce that it is an isomorphism. By a similar argument, since $p^eK_Y + p^eE_- \sim \pi^*(p^eK_X) +p^eE_+$, we deduce that $\pi_*\cO_Y(p^eK_Y + p^eE_-) \cong \cO_X(p^eK_X)$. Note that $\cP_1 \cong \cO_Y(p^eK_Y + p^eE_- + E' - \lfloor E_- \rfloor)$ by \autoref{pushout_sucks_but_pullback_vibes}, so there is a strict inclusion $\pi_*\cP_1 \inc \cO_X(p^eK_X)$ by \autoref{BPRZ} since $E' - \lfloor E_- \rfloor$ is negative. 
	%	\[ \pi_*\cP_1 \cong \pi_*\cO_Y(\pi^*(p^eK_X) + p^eE_+ + (E' - \lfloor E_- \rfloor )) \cong \cO_X(p^eK_X) \otimes \pi_*\cO_Y(p^eE_+ + (E' - \lfloor E_- \rfloor)). \] The very same argument as in the proof of \autoref{BPRZ}. gives that $\pi_*\cO_Y(p^eE_+ + (E' - \lfloor E_- \rfloor)) \inc \cO_X$ is a strict ideal, to the inclusion $\pi_*\cP_1 \inc \cO_X(p^eK_X)$ is not surjective. 
	Finally, given that $\pi_*W_n\omega_{(Y, \lfloor E_- \rfloor))}((p^e - 1)K_Y + p^eE_- - \lfloor E_- \rfloor)$ is torsion-free and is equal to the $S_2$-sheaf $W_n\omega_X((p^e - 1)K_X)$ on a big open subset, we obtain a natural inclusion \[ \pi_*W_n\omega_{(Y, \lfloor E_- \rfloor))}((p^e - 1)K_Y + p^eE_- - \lfloor E_- \rfloor) \inc W_n\omega_X((p^e - 1)K_X). \] Since $\pi_*$ preserves pullbacks (it is left exact), we obtain by pushing forward and our discussion above a commutative diagram 
	\[ \begin{tikzcd}
		F^e_*W_n\omega_X(p^e(p^e - 1)K_X) \arrow[dd]                                                              &  & \pi_*\cP_2 \arrow[ll, hook'] \arrow[d] \\
		&  & \pi_*\cP_1 \arrow[d, "\neq", hook]     \\
		{\pi_*W_n\omega_{(Y, \lfloor E_- \rfloor)}((p^e - 1)K_Y + p^eE_- - \lfloor E_- \rfloor)} \arrow[d, hook'] &  & \cO_X(p^eK_X) \arrow[ll, hook']     \\
		W_n\omega_X((p^e - 1)K_X),                                                                                 &  &                                     
	\end{tikzcd} \] where the big square is a pullback diagram. In particular, the map $\pi_*\cP_2 \to \cO_X(p^eK_X)$ is not surjective. Let us now see why this contradicts the $n$-quasi-$F^e$-purity of $X$. 
	
	First, it follows from the definition of a pullback that $\pi_*\cP_2$ is also the pullback of the diagram 
	\[ \begin{tikzcd}
		F^e_*W_n\omega_X(p^e(p^e - 1)K_X) \arrow[d] &  &                                  \\
		W_n\omega_X((p^e - 1)K_X)                   &  & \cO_X(p^eK_X). \arrow[ll, hook']
	\end{tikzcd} \] Since $X$ is $n$-quasi-$F^e$-pure, we know by \cite[Proposition 3.19]{Quasi_F^e_splittings_and_quasi-F_regularity} (after twisting by $W_n\cO_X((p^e - 1)K_X)$ and extend from the regular locus of $X$ using the $S_2$ property) that there is a commutative diagram 
	\[ \begin{tikzcd}
		F^e_*W_n\omega_X(p^e(p^e - 1)K_X) \arrow[d] &  &                                  &  &                                           \\
		W_n\omega_X((p^e - 1)K_X)                   &  & \cO_X(p^eK_X) \arrow[ll, hook'] &  & W_n\cO_X(p^eK_X). \arrow[ll] \arrow[llllu]
	\end{tikzcd} \] By definition of a pullback, we therefore obtain a commutative diagram \[ \begin{tikzcd}
	F^e_*W_n\omega_X(p^e(p^e - 1)K_X) \arrow[d] &  & \pi_*\cP_2 \arrow[d] \arrow[ll]    &  &                                         \\
	W_n\omega_X((p^e - 1)K_X)                   &  & \cO_X(p^eK_X) \arrow[ll, hook'] &  & W_n\cO_X(p^eK_X). \arrow[ll] \arrow[llu]
	\end{tikzcd} \] Since the map $W_n\cO_X(p^eK_X) \to \cO_X(p^eK_X)$ is surjective (a set-theoretic lift being for example the Teichmüller lift), we deduce that $\pi_*\cP_2 \to \cO_X(p^eK_X)$ must also be surjective, contradicting what we previously established.
\end{proof}

%\begin{scor}
%	Let $X$ be a normal, quasi-$F^{\infty}$-pure and $\bQ$-Gorenstein threefold. Then $X$ has log canonical singularities.
%\end{scor} 

\bibliographystyle{alpha}
\bibliography{bibliography}

\Addresses

\end{document}

%% file: macros.tex
\documentclass[a4paper,12pt]{amsart}

\usepackage{environ}
\usepackage{fancyhdr}
\usepackage{mabliautoref}
\usepackage{amssymb,amsthm,amsmath}
\RequirePackage[dvipsnames,usenames]{xcolor}
\usepackage{hyperref}
\usepackage{cleveref}
\usepackage{mathtools}
\usepackage{tcolorbox}
\usepackage[all]{xy}
\usepackage{tikz}
\usepackage{tikz-cd}

\makeatletter
\def\fixtikzforbreqn#1#2{%
	\protected\edef#1{\noexpand\ifmmode\mathchar\the\mathcode`#2 \noexpand\else#2\noexpand\fi}%
}
\fixtikzforbreqn\tikz@nonactivesemicolon;
\fixtikzforbreqn\tikz@nonactivecolon:
\fixtikzforbreqn\tikz@nonactivebar|
\fixtikzforbreqn\tikz@nonactiveexlmark!
\makeatother

\usetikzlibrary{decorations.markings,shapes,positioning}
\usepackage{enumitem}
\usepackage{stmaryrd}
\usepackage{xfrac}
\usepackage{fix-cm}
\usepackage{faktor}
\usepackage{minibox}
\usepackage{commath}
\usepackage{chngcntr}
\usepackage{setspace}
\usepackage{array}
\usepackage[toc,page]{appendix}
\usepackage{calligra,mathrsfs}
\usepackage{mathtools}
\usepackage{bbm}
\hypersetup{
bookmarks,
bookmarksdepth=3,
bookmarksopen,
bookmarksnumbered,
pdfstartview=FitH,
colorlinks,backref,hyperindex,
linkcolor=Sepia,
anchorcolor=BurntOrange,
citecolor=Cyan,
filecolor=BlueViolet,
menucolor=Yellow,
urlcolor=OliveGreen
}

\newtcolorbox{activitybox}[1][]{%
	breakable,
	enhanced,
	colback=lightergray,
	boxrule=3pt,
	arc=5pt,
	outer arc=5pt,
	boxsep=10pt,
	colframe=darkergray,
	coltitle=white,
	#1
}

\newcommand{\quot}[2]{%
	\raise1ex\hbox{$#1$}\Big/\lower1ex\hbox{$#2$}%
}

\renewcommand{\lim}{\varprojlim}

\makeatletter\def\Rvarlim@#1#2{%
	\vtop{\m@th\ialign{##\cr
			\hfil$#1\operator@font Rlim$\hfil\cr
			\noalign{\nointerlineskip\kern1.5\ex@}#2\cr
			\noalign{\nointerlineskip\kern-\ex@}\cr}}%
}
\makeatother

\makeatletter\def\Rlim{%
	\mathop{\mathpalette\Rvarlim@{\leftarrowfill@\textstyle}}\nmlimits@
}
\makeatother

\newcommand{\expl}[2]{\underset{\mathclap{\minibox[c]{$\uparrow$\\ \fbox{\footnotesize #2}}}}{#1}}

\newcommand{\inj}{\hookrightarrow}

\newcommand{\HHom}{\mathcal{H}om}

\newcommand{\stacksproj}[1]{{\cite[Tag~\href{http://stacks.math.columbia.edu/tag/#1}{#1}]{Stacks_Project}}}

\newcommand{\stacksprojs}[2]{\cite[Tags ~\href{http://stacks.math.columbia.edu/tag/#1}{#1} and ~\href{http://stacks.math.columbia.edu/tag/#2}{#2}]{Stacks_Project}}

\newcommand{\inc}{\subseteq}

\newcommand{\cni}{\supseteq}

\DeclareMathAlphabet{\mathchanc}{OT1}{pzc}%
                                 {m}{it}

\renewcommand{\set}[2]{\{ \ #1 \  | \ #2 \ \}}

\newcommand{\bD}{\mathbb{D}}

\newcommand{\bN}{\mathbb{N}}

\newcommand{\bP}{\mathbb{P}}
\newcommand{\bQ}{\mathbb{Q}}

\newcommand{\bZ}{\mathbb{Z}}

\newcommand{\scr}{\mathcal}

\newcommand{\cH}{\scr{H}}
\newcommand{\cI}{\scr{I}}
\newcommand{\cJ}{\scr{J}}

\newcommand{\cM}{\scr{M}}

\newcommand{\cO}{\scr{O}}
\newcommand{\cP}{\scr{P}}

\newcommand{\cR}{\scr{R}}

\newcommand{\cX}{\scr{X}}

\DeclareMathOperator{\codim}{codim}

\DeclareMathOperator{\dR}{dR}

\DeclareMathOperator{\im}{{im}}

\DeclareMathOperator{\length}{{length}}

\DeclareMathOperator{\red}{red}

\DeclareMathOperator{\reg}{reg}

\DeclareMathOperator{\Spec}{{Spec}}

\DeclareMathOperator{\Supp}{{Supp}}

\DeclareMathOperator{\crys}{crys}

\DeclareMathOperator{\ShHom}{\mathscr{H}\text{\kern -3pt {\calligra\large om}}\,}

\DeclareFontFamily{OT1}{pzc}{}
\DeclareFontShape{OT1}{pzc}{m}{it}{<-> s * [1.200] pzcmi7t}{}
\DeclareMathAlphabet{\mathpzc}{OT1}{pzc}{m}{it}

\AtBeginDocument{%
	\mathchardef\phialt=\phi
	\mathchardef\phi=\varphi
}

\newcommand{\factor}[2]{\left. \raise 2pt\hbox{\ensuremath{#1}} \right/
        \hskip -2pt\raise -2pt\hbox{\ensuremath{#2}}}

\makeatletter
\renewcommand\subsection{
  \renewcommand{\sfdefault}{pag}
  \@startsection{subsection}%
  {2}{0pt}{.8\baselineskip}{.4\baselineskip}{\raggedright
    \sffamily\itshape\small\bfseries
  }}
\renewcommand\section{
  \renewcommand{\sfdefault}{phv}
  \@startsection{section} %
  {1}{0pt}{\baselineskip}{.8\baselineskip}{\centering
    \sffamily
    \scshape
    \bfseries
}}
\makeatother

\definecolor{gr}{rgb}{0,0.5,0}

\newcommand{\Addresses}{{% additional braces for segregating \footnotesize
		\bigskip
		\footnotesize
		
		\textsc{\'Ecole Polytechnique F\'ed\'erale de Lausanne, SB MATH CAG, MA C3 615 (B\^atiment MA), Station 8, CH-1015 Lausanne, Switzerland}\par\nopagebreak
		\textit{E-mail address}: \texttt{jefferson.baudin@epfl.ch}

}}